\documentclass[a4paper, 12pt]{article}

\usepackage[utf8]{inputenc}
\usepackage[T1]{fontenc}
\usepackage{cite}
\usepackage[a4paper, margin=2.6cm]{geometry}
\def\cqedsymbol{\ifmmode$\lrcorner$\else{\unskip\nobreak\hfil
\penalty50\hskip1em\null\nobreak\hfil$\lrcorner$
\parfillskip=0pt\finalhyphendemerits=0\endgraf}\fi}

\title{Every graph is eventually Tur\'an-good}
\author{Natasha Morrison%
\thanks{Department of Mathematics and Statistics, University of Victoria, David Turpin Building,
3800 Finnerty Road, Victoria, BC, Canada V8P 5C2. E-mail: \texttt{nmorrison@uvic.ca}.
\newline Research supported by NSERC Discovery Grant RGPIN-2021-02511 and NSERC Early Career Supplement DGECR-2021-00047.} \and  JD Nir%
\thanks{Department of Mathematics, Toronto Metropolitan University, Toronto, ON, Canada.
E-mail: \texttt{jd.nir@ryerson.ca}.} 
\and Sergey Norin% 
\thanks{Department of Mathematics and Statistics, McGill University, Montr\'{e}al, QC, Canada. E-mail: \texttt{snorine@gmail.com}.
Research supported by NSERC Discovery Grant.}
\and Paweł Rzążewski%
\thanks{Warsaw University of Technology \& University of Warsaw, Warsaw, Poland E-mail: \texttt{pawel.rzazewski@pw.edu.pl}.
\newline This research is a part of projects that have received funding from the European Research Council (ERC) under the European Union's Horizon 2020 research and innovation programme
Grant Agreement 948057.}
\and Alexandra Wesolek% 
\thanks{Department of Mathematics, Simon Fraser University, Burnaby, BC, Canada. E-mail:
\texttt{alexandrawesolek@gmail.com}.
\newline Research supported by the Vanier Canada Graduate Scholarships program.
}
}

\date{\today}

\newcommand\ii{\mathsf{inj}}
\newcommand\mc[1]{\mathcal{#1}}
\newcommand{\II}{\mathsf{INJ}}
\newcommand{\e}{\textup{\textsf{e}}}
\renewcommand{\v}{\textup{\textsf{v}}}
\newcommand{\ex}{\textsf{ex}}
\newcommand{\s}[1]{\left(#1\right)}
\newcommand\Ed{\e(T_r(n))}

\usepackage{amsmath, amsthm, amssymb}
\usepackage{thmtools}
\usepackage{thm-restate}
\declaretheorem[name=Theorem,     refname={Theorem,Theorems},         numberwithin=section]{theorem}
\declaretheorem[name=Lemma,       refname={Lemma,Lemmas},             sibling=theorem]{lemma}

\declaretheorem[name=Corollary,   refname={Corollary,Corollaries},    sibling=theorem]{corollary}

\declaretheorem[name=Claim,     refname={Claim,Claims}, numberwithin=theorem]{claim}

\declaretheorem[name=Question,     refname={Question,Questions},         sibling=theorem]{question}
\newenvironment{claimproof}{\noindent {\emph{Proof of Claim.}}}{\hfill\cqedsymbol\medskip}

\begin{document}
	\maketitle
\begin{abstract}
	Let  $H$  be a graph. We show that if $r$ is large enough as a function of $H$, then the $r$-partite Tur\'an graph maximizes the number of copies of $H$ among all $K_{r+1}$-free graphs on a given number of vertices. This confirms a conjecture of Gerbner and Palmer. 
\end{abstract}

\section{Introduction}

For a pair of graphs $G$ and $F$, say that $G$ is \emph{$F$-free} if $G$ does not contain a subgraph isomorphic to $F$. 
Let $\mc{N}(H,G)$ denote the number of copies of a graph $H$ in $G$, that is, the number of subgraphs of $G$ isomorphic to $H$, and let 
\[\ex(n,H,F) = \max\left\{\mc{N}(H,G) \: | \: G \; \text{is an} \; n\text{-vertex} \; F\text{-free graph}\right\}\]
be the maximum number of subgraphs isomorphic to the target graph $H$ in an $n$-vertex $F$-free graph.

For the case $H = K_2$, this function has been widely studied. Indeed, the classical theorem of Tur\'{a}n~\cite{T41} states that the unique $n$-vertex $K_{r+1}$-free graph with $\ex(n,K_2,K_{r+1})$ edges is the \emph{Tur\'an graph $T_r(n)$}: the complete $r$-partite graph with parts of size either $\lfloor\frac{n}{r}\rfloor$ or $\lceil\frac{n}{r}\rceil$. The special case $r=2$ was originally resolved by Mantel~\cite{M1907} in 1907. The large scale behaviour of $\ex(n,K_2,F)$ was resolved by the powerful Erd\H{o}s-Stone-Simonovits Theorem~\cite{ES66, ES46}. This theorem gives asymptotically tight bounds for $\ex(n,K_2,F)$ in terms of the chromatic number of $F$, for all non-bipartite graphs $F$, and shows the Tur\'{a}n graph is essentially best possible.

\begin{theorem}[Erd\H{o}s-Stone-Simonovits~\cite{ES46, ES66}]
Let $F$ be a graph with chromatic number $k \ge 2$. Then
\[\ex(n,K_2,F) = (1+o(1))\e(T_{k-1}(n)).\]
\end{theorem}

Following a number of 
results concerning $\ex(n,H,F)$ when $H \not= K_2$ (see, for example,~\cite{BG08,B76,E62,HKNR13}), the systematic study of this function in the more general setting was initiated by Alon and Shikhelman~\cite{AS15}. This has sparked a number of directions of research and bounds on $\ex(n,H,F)$  are now known for many families of $H$ and $F$, but exact results are rare. (See e.g.~\cite{GP22} for additional details.)

One difficulty in precisely determining $\ex(n,H,F)$ is identifying a potential \emph{extremal graph}, that is an $n$-vertex $F$-free graph $G$ such that  $\ex(n,H,F)=\mc{N}(H,G)$. In many cases when the problem is tangible, the extremal graph turns out to be the Tur\'an graph. Let $F$ be a graph with chromatic number $k+1$ and say that a graph $H$ is \emph{$F$-Tur\'an-good} if for $n$ sufficiently large, $\ex(n,H,F) = \mc{N}(H,T_{k}(n)))$. That is, the Tur\'{a}n graph $T_{k}(n)$ is an $n$-vertex $F$-free graph containing the maximum possible number of copies of $H$. The term \emph{Tur\'an-good} was recently introduced by Gerbner and Palmer~\cite{GP22}, but the study of this phenomenon goes back much further to work of Gy\"{o}ri, Pach and Simonovits~\cite{GPS91}. See~\cite{GP22} for a comprehensive summary of what is known so far about $F$-Tur\'{a}n-good graphs.

Gerbner and Palmer~{\cite[Conjecture 20]{GP22}} conjectured that for every graph $H$ there exists $r_0 = r_0(H)$ such that $H$ is $K_{r+1}$-Tur\'an-good for every $r \geq r_0(H)$. This conjecture is known to hold for stars~\cite{CNR22}, more generally for complete multipartite graphs~\cite{GP22}, for paths~\cite{G22} and for the $5$-cycle~\cite{LM21}.
 In this paper we prove the conjecture holds with $r_0 = 300\v(H)^9$.
 
\begin{theorem}\label{thm:main1}
Let $H$ be a graph and $r \geq 300\v(H)^9$. Then $H$ is $K_{r+1}$-Tur\'{a}n-good. 
 \end{theorem}

 Theorem~\ref{thm:main1} follows from a more technical result (Theorem~\ref{thm:main2} below, which also implies that for $r \geq 300\v(H)^9$ Tur\'an graphs always maximize the number of copies of $H$ among $K_{r+1}$-free graphs on any given number of vertices, i.e.,~the requirement that the number of vertices is large compared to $r$ is unnecessary). 

In~\cite{LM21} it is conjectured that Theorem~\ref{thm:main1} should hold with $r_0(H)=\v(H)+1$. Counterexamples are known for all $r \le \chi(H)+1$, but determining the exact value of $r_0(H)$ remains open, see~\cite{G22b, GGST22, KM22, LM21}. We did not attempt to optimize our bound, beyond ensuring that it is polynomial in $\v(H)$, and prioritized simplicity of presentation. However,   our methods are unlikely to extend to this stronger conjecture. Even a bound quadratic in $\v(H)$ would likely require additional ideas.

\subsection{Preliminary definitions and proof outline}

Throughout the paper we use $\v(G)$ and $\e(G)$ to denote the number of vertices and edges, respectively, in a graph $G$. It will be more convenient for us to work with homomorphism counts rather than subgraphs, and we precede the statement of Theorem~\ref{thm:main2} with the necessary definitions. 

Recall that we say a map $\varphi: V(H) \to V(G)$ is \emph{a homomorphism from $H$ to $G$} if $\varphi(u)\varphi(v) \in E(G)$ for every $uv \in E(H)$. 
We denote by $\II(H,G)$ the set of all injective homomorphisms from $H$ to $G$, and let $\ii(H,G)= |\II(H,G)|$. 

Clearly, $\ii(H,G) = |\text{Aut}(H)|\cdot\mc{N}(H,G)$, where $\text{Aut}(H)$ is the automorphism group of $H$, and so a graph $H$ is $K_{r+1}$-Tur\'an-good if and only if $\ii(H,G) \leq \ii(H,T_r(n))$ for  every $n$-vertex $K_{r+1}$-free graph $G$ (and $n$ is sufficiently large). 

Theorem~\ref{thm:main1} is a direct consequence of the following theorem.

\begin{theorem}
\label{thm:main2}
For every graph $H$, every $r \geq 300\v(H)^9$, and every $n$-vertex $K_{r+1}$-free graph $G$ we have
$$
    \ii(H,G)\leq \ii(H,T_r(n)).
$$
\end{theorem}

\subsection{Proof Outline}

Here we present an informal outline of the proof, which will be given in Section~\ref{sec:proof}. As the theorem trivially holds when $n \le r$, we can assume that $n \geq r+1 \geq 300\v(H)^9$, and so $n$ is much larger than $\v(H)$. Thus almost every $\varphi: V(H) \to V(T_r(n))$  is in $\II(H,T_r(n))$ and most of the maps that are not injective homomorphisms fail due to a single edge of $H$ being mapped to a non-edge of $T_r(n)$, i.e.,~there are approximately $\e(H)(n^2-2\e(T_r(n)))n^{\v(H)-2}$ failures. Similarly, if an $n$-vertex $K_{r+1}$-free graph $G$ maximizes $\ii(H,G)$, approximately $\e(H)(n^2-2\e(G))n^{\v(H)-2}$ maps fail to be an injective homomorphism from $H$ to $G$. As $\e(T_r(n)) \geq \e(G)$, Theorem~\ref{thm:main2} would follow as long as we can control the error terms implicit in the term ``approximately''. 

Arguing along the lines of the above sketch, we additionally see that adding any edge to $G$ increases $\ii(H,G)$ and removing any edge decreases $\ii(H,G)$ by approximately the same amount. If $G$ could be transformed to $T_r(n)$ by changing $O\bigl(\e(T_r(n)) -\e(G)\bigr)$ adjacencies 
then the naive estimates hinted at above would suffice to make the argument precise. Unfortunately, such a transformation is not always possible. However, a sharp stability version of Tur\'an's theorem due to F\"uredi~\cite{furedi2015proof} allows us to transform $G$ into a (not necessarily balanced) $r$-partite graph $G_0$.

\begin{theorem}[F\"uredi~\cite{furedi2015proof}]
\label{thm:Furedi}
Every $n$-vertex $K_{r+1}$-free graph $G$ contains an $r$-partite subgraph $G_0$ such
that $$\e(G) - \e(G_0) \leq \e(T_r(n)) - \e(G).$$
\end{theorem}

Given this, we subsequently transform the $r$-partite graph $G_0$ into $T_r(n)$, carefully controlling the homomorphism count, which gives the required result.

The key technical lemmas for our proof will be given in Section~\ref{sec:lemmas}, before using them to complete the proof of Theorem~\ref{thm:main2} in Section~\ref{sec:proof}.

\section{Density Lemmas}\label{sec:lemmas}
The goal of this section is to prove Lemma~\ref{l:rpartite}, which bounds the difference between $\ii(H,T_r(n))$ and $\ii(H,G)$ for an $r$-partite graph $G$, by a function depending on the `density' of $G$. To capture the notion of density that will be useful to us, we define
a graph $G$ to be \emph{$\delta$-dense} for some $\delta >0$ if  $\deg(v) \geq (1-\delta)\v(G)$  for every $v \in V(G)$. 

We begin by showing that if $\delta$ is small enough, then in every $\delta$-dense graph every partial injective homomorphism has almost the maximum number of extensions.

\begin{lemma}\label{l:lowerbound}
	Let $G$ and $H$ be graphs such that $G$ is $\delta$-dense. Let $H' = H[X]$  for some $X \subseteq V(H)$. Let $k=\v(H)-|X|$. Then for every $\varphi' \in \II(H',G)$ there exist at least $$\bigl(1 - \delta k\cdot\v(H)\bigr) \cdot \v(G)^{k}$$ injective homorphisms  $\varphi \in \II(H,G)$ such that  $\varphi|_X=\varphi'.$
\end{lemma}

\begin{proof}
	The proof proceeds by induction on $k$. The case $k=0$ is trivial. 
	
	Suppose now $k=1$ and let $V(H) \setminus X = \{u\}$. By the degree condition  at least $\bigl(1 - \delta\v(H)\bigr)\v(G)$ vertices are adjacent to every vertex in $\varphi'(V(H'))$ and setting $\varphi(u)=v$ for any such vertex $v$ extends $\varphi'$ to an  injective homomorphism  $\varphi$ from $H$ to $G$, as desired.
	
	The induction step follows readily from the base case. Let $u \in V(H)-X$ be arbitrary and let $H''=H[X \cup \{u\}]$ 
	By the base case, for every $\varphi' \in \II(H',G)$ there exist at least $\bigl( 1-\delta\v(H'') \bigr) \cdot \v(G)$ maps $\varphi'' \in \II(H'',G)$ such that $\varphi''|_X=\varphi'.$ By the induction hypothesis for every such $\varphi''$ there are at least
	$$\bigl(1 - \delta (k-1)\v(H)\bigr) \cdot \v(G)^{k-1}$$  injective homomorphisms  $\varphi \in \II(H,G)$ such that  $\varphi|_{X \cup \{u\}}=\varphi''.$ Thus we get at least 
	\begin{align}
	\bigl( 1-\delta\v(H'') \bigr) \cdot \v(G) \cdot \bigl(1 - \delta (k-1)\v(H)\bigr) \cdot \v(G)^{k-1} \geq \bigl(1 - \delta k\v(H)\bigr) \cdot \v(G)^{k}    
	\label{eq:induction}
	\end{align}
	injective homomorphisms $\varphi$ from $H$ to $G$ extending $\varphi'$ as desired.
\end{proof}

Setting $X = \emptyset$ in Lemma~\ref{l:lowerbound}, i.e., letting $k = \v(H)$, yields the following corollary.

\begin{corollary}\label{l:iJG}
	If $G$ and $H$ are graphs such that $G$ is $\delta$-dense,
	then $$\ii(H,G) \geq \bigl( 1 - \delta \v(H)^2 \bigr) \cdot \v(G)^{\v(H)}.$$
\end{corollary}

The next lemma is the key technical step in the proof of Lemma~\ref{l:rpartite}, controlling the change of homomorphism counts in a multipartite graph as we rebalance the sizes of the parts.

\begin{lemma}\label{l:rebalance}
	Let $H$ be a graph with at least one edge. Let $0 < \delta  \leq 1/4$, let $G$ be a $\delta$-dense graph, and let $A,B \subseteq V(G)$ be a pair of disjoint independent sets such that \begin{itemize}
		\item $|A| \geq |B| \geq 1$,
		\item every vertex in $A$ is adjacent to every vertex in $V(G)-A$,  and
		\item every vertex in $B$ is adjacent to every vertex in $V(G)-B$.
	%	 \item $G \setminus A \setminus B$ is $\delta$-dense.
		 \end{itemize} Let $G_A$ and $G_B$ be obtained from $G$ by deleting one vertex in $A$ and $B$, respectively.
Then \[\ii(H,G_A) \geq \ii(H,G_B) + 2\e(H)(|A|-|B|)\bigl(1 - 3\delta\v(H)^3\bigr) \cdot \v(G)^{\v(H)-2}.\]
\end{lemma}
\begin{proof}
Let $G_0=G\setminus (A \cup B)$ and $G_1 = G[A \cup B]$. Let $a \in A$ and $b \in B$ be such that $G_A = G\setminus a$  and  $G_B = G\setminus b$. 
	
	For $S \subseteq V(H)$ and $G' \in \{G_A,G_B\}$, let $\II_{S}(H,G')$ denote the set of injective homomorphisms $\varphi$ from $H$ to $G'$ such that $\varphi^{-1}(A \cup B)=S$. Then 
	$$|\II_{S}(H,G')| = \ii(H[S],G'[A \cup B]) \cdot \ii(H \setminus S, G_0).$$
	Let
	\begin{align}
	\begin{split}
	\Delta(S)=& |\II_S(H,G_A)|-|\II_S(H,G_B)| \\
	= &\s{\ii(H[S],G_1 \setminus a) -\ii(H[S],G_1 \setminus b)}\cdot \ii(H \setminus S, G_0).  \label{e:defDelta}
	\end{split}
	\end{align}
	Then \begin{equation}\label{e:Delta0} \ii(H,G_A) - \ii(H,G_B) = \sum_{S \subseteq V(G)}\Delta(S), \end{equation}
	and so it suffices to lower bound $\sum_{S \subseteq V(G)}\Delta(S)$.
	
	\medskip
	 %moved this paragraph down
	  First let us lower bound the value $\ii(H \setminus S, G_0)$, which appears in \eqref{e:defDelta}.
	  
	  \begin{claim}
	 $\ii(H \setminus S, G_0) \geq \bigl(1 - 2\delta \v(H)^2\bigr) \cdot \v(G_0)^{\v(H)-|S|}.$ \label{e:iiSGo}
	  \end{claim}
	  \begin{claimproof}
	  We aim to use Corollary~\ref{l:iJG}, so we need to calculate the density of $G_0$.
	 Since $A$ is an independent set, $\v(G)-|A| \geq \deg(v) \geq (1-\delta)\v(G)$ for every $v \in A$, and thus i$|A| \leq \delta\v(G)$. Similarly, $|B| \leq \delta\v(G)$ and so
	 \begin{equation}
	 \v(G_0) \geq (1-2\delta)\v(G).
	 \label{e:sizeG0}
	 \end{equation}
	 In particular, $\v(G_0) \geq \v(G)/2$.
	 As every vertex of $G_0$ has at most $\delta\v(G) \leq 2\delta\v(G_0)$ non-neighbors in $G_0$, the graph $G_0$ is $2\delta$-dense.
	 Consequently, we obtain the claim by Corollary~\ref{l:iJG}.
	 \end{claimproof}

Before we estimate $\Delta(S)$, let us show one more technical claim.

\begin{claim}\label{clm:xyp}
Let $x,y$ be non-negative reals and $p$ be a non-negative integer.
Then $(1-x)(1-y)^p \geq 1-x-py$.
\end{claim}
\begin{claimproof}
Induction on $p$. If $p=0$, then the claim is trivial.
Thus suppose that $p \geq 1$ and $(1-x')(1-y')^{p-1} \geq 1-x'-(p-1)y'$ for all $x',y' \geq 0$. We have
\begin{align*}
(1-x)(1-y)^p = &(1-x)(1-y) \cdot (1-y)^{p-1} = (1-x-y+xy) \cdot (1-y)^{p-1} \\
\geq & (1-x-y)(1-y)^{p-1} \geq  1-x-py.
\end{align*}
where the first inequality follows since $xy \geq  0$ and the second inequality follows by the inductive assumption for $x'=x+y$ and $y'=y$.
\end{claimproof}

Now let us estimate $\Delta(S)$  when $|S| \leq 2$.

\begin{claim}\label{clm:DeltaS2}
$\sum_{S \subseteq V(H), |S| \leq 2} \Delta(S) \geq 2\e(H)(|A|-|B|)(1 - 2\delta\v(H)^3) \cdot \v(G)^{\v(H)-2}.$
\end{claim}
    \begin{claimproof}
        If $S$ is independent in $H$, then $\Delta(S)=0$. 
	If $|S|=2$ and the vertices of $S$ are adjacent, then a homomorphism from $H[S]$ to $G_1\setminus a$ or $G_1\setminus b$ maps $S$ to an edge, hence
	
\begin{align*}
\Delta(S)& =  \s{2\e(G_1 \setminus a)- 2\e(G_1 \setminus b)} \cdot \ii(H \setminus S, G_0) \\
& =  2\s{(|A|-1)|B| - 2(|B|-1)|A|} \cdot \ii(H \setminus S, G_0) \\
& =  2(|A|-|B|) \cdot \ii(H \setminus S, G_0).
\end{align*}	  
By Claim~\ref{e:iiSGo} and \eqref{e:sizeG0} we further get:
\begin{align*}
\Delta(S) & \geq   2(|A|-|B|)(1 - 2\delta \v(H)^2) \cdot \v(G_0)^{\v(H)-2} \\
& \geq   2(|A|-|B|)(1 - 2\delta \v(H)^2) \bigl((1-2\delta) \cdot \v(G)\bigr)^{\v(H)-2} \\
& =  2(|A|-|B|)(1 - 2\delta \v(H)^2) (1-2\delta)^{\v(H)-2} \cdot \v(G)^{\v(H)-2}.
\end{align*}	  
Now, applying Claim~\ref{clm:xyp} for $x=2\delta \v(H)^2$, $y=2\delta$, and $p = \v(H)-2$, we get
\begin{align*}	  
\Delta(S) & \geq 2(|A|-|B|)(1 - 2\delta \v(H)^2 - 2\delta(\v(H)-2)) \cdot \v(G)^{\v(H)-2}\\ 
	  & \geq  2(|A|-|B|)(1 - 2\delta \v(H)^3) \cdot \v(G))^{\v(H)-2}.
	  \end{align*}
Thus
\[
	\sum_{S \subseteq V(H), |S| \leq 2} \Delta(S) \geq 2\e(H)(|A|-|B|)(1 - 2\delta\v(H)^3) \cdot \v(G)^{\v(H)-2},
\]
as claimed.
	  \end{claimproof}
	  
 It remains to lower bound $\sum_{S \subseteq V(H), |S| \geq 3} \Delta(S)$.
 
\begin{claim}\label{clm:DeltaS3}
 $\sum_{S \subseteq V(H), |S| \geq 3} \Delta(S) \geq -2\e(H)(|A|-|B|)\cdot \delta \v(H)^3 \cdot \v(G)^{\v(H)-2}.$
\end{claim} 
 \begin{claimproof}
In order to prove the claim, we construct an injection from a large subset of $\II(H,G_B)$ into $\II(H,G_A)$ that maps homomorphisms in $\II_{S}(H,G_B)$ to homomorphisms in $\II_{S}(H,G_A)$, and lower bound the sum above by $-1$ times the number of remaining homomorphisms in $\bigcup_{|S| \geq 3} \II_S(H,G_B)$.  For example, if $\varphi$ is a homomorphism from $H$ to $G_B$, such that $a\notin \varphi(H)$, then $\varphi$ is also a homomorphism from $H$ to $G_A$, so we get a natural correspondence. If $a\in \varphi(H)$, then $\varphi$ is not a map from $H$ to $G_A$, hence we will define a correspondence in the following (for most such maps $\varphi$).
	  	 
First, choose any $A' \subseteq A$ such that $a \in A' $ and $|A'|=|B|$. Let $\iota: (A' \cup B) \to (A' \cup B)$ be an involution corresponding to a perfect matching between $A'$ and $B$ containing the edge $ab$, 
     that is $\iota(a)=b$, $\iota(A')=B$ and $\iota(\iota(v))=v$ for every $v \in A' \cup B$.
	  	 We extend $\iota$ to $V(G)-(A-A')$ by setting $\iota(v)=v$ for every $v \in V(G_0)$.
	  	 Define $\II^{*}(H,G_B) \subseteq \II(H,G_B) $  as the set of all injective homomorphisms $\varphi$ from $H$ to $G_B$ such that either \begin{itemize} \item $a \not \in \varphi(V(H))$, or
	  	 	\item $\varphi(V(H)) \cap (A-A') = \emptyset.$ 
	  	 	\end{itemize}
  	 	The map $f:\II^{*}(H,G_B) \to \II(H,G_A)$ is defined by setting  $f\varphi = \varphi$ if  $a \not \in \varphi(V(H))$,
  	 	and $f\varphi = \iota\varphi$ if $a \in \varphi(V(H))$. Note that $b \notin \varphi(V(H))$ hence $a \notin \iota \varphi(V(H))$, and in  the second case,  $\varphi(V(H)) \cap (A-A') = \emptyset$ and so $\iota$ is indeed well defined on $\varphi(V(H))$.
  	 	It is easy to see that $f$ has the properties we stated in the previous paragraph, i.e.,~$f$ is an injection and $f$ maps $\II^{*}(H,G_B) \cap \II_{S}(H,G_B)$ into $\II_{S}(H,G_A)$ for every $S \subseteq V(H)$. It follows that $$\Delta(S) = |\II_S(H,G_A)|-|\II_S(H,G_B)| \geq -| \II_S(H,G_B) - \II^{*}(H,G_B)|,$$ for every $S \subseteq V(H)$, and so 
  	 	\begin{equation}\label{e:Delta2}\sum_{S \subseteq V(H), |S| \geq 3} \Delta(S) \geq -\left|\s{\bigcup_{|S| \geq 3}  \II_S(H,G_B)} - \II^{*}(H,G_B)\right|. \end{equation}
  	 	Every homomorphism in $\bigcup_{|S| \geq 3}  \II_S(H,G_B) - \II^{*}(H,G_B)$ maps some vertex of $H$ to $a$, 
  	 	another vertex to $A-A'$ and at least one more vertex to $A \cup B$.
  	 	Thus\begin{align}\label{e:Delta3} &\left|\s{\bigcup_{|S| \geq 3}  \II_S(H,G_B)} - \II^{*}(H,G_B)\right|\notag \\ &\qquad \leq \v(H)\cdot (|A|-|B|)\v(H) \cdot (|A|+|B|)\v(H) \cdot \v(G)^{\v(H)-3} \notag\\ &\qquad \leq (|A|-|B|)\v(H)^3 \cdot 2\delta\v(G) \cdot \v(G)^{\v(H)-3}\notag\\ & \qquad \leq 2\e(H)(|A|-|B|)\cdot \delta \v(H)^3 \cdot \v(G)^{\v(H)-2}.\end{align}
Combining \eqref{e:Delta2} and \eqref{e:Delta3} completes the proof of the claim.
\end{claimproof}  	 	
  	 	
Now the statement of the lemma follows directly from equation \eqref{e:Delta0} and Claims~\ref{clm:DeltaS2} and \ref{clm:DeltaS3}.
\end{proof}

Next we prove that a strengthening of Theorem~\ref{thm:main2} holds if $G$ is $r$-partite, i.e.,~in this case we not only show that $\ii(H,G) \leq  \ii(H,T_r(n)),$ but  give an essentially optimal bound on $\ii(H,T_r(n)) - \ii(H,G)$ in terms of $\Ed - \e(G)$.

\begin{lemma}
\label{l:rpartite}
For every graph $H$ with at least one edge, every {$0 < \delta \leq 1/4$}, and every $r$-partite $\delta$-dense graph $G$ we have
   \begin{equation}\label{e:rpartite}
\ii(H,T_r(n)) - \ii(H,G) \geq 2\e(H)(1-3\delta\v(H)^3) \left(\Ed - \e(G)\right)n^{\v(H)-2}.
   \end{equation} 
\end{lemma}
\begin{proof}
We prove the lemma by induction on the difference $\Ed - \e(G)$.
By Tur\'an's theorem$~{\e(G) \leq \Ed}$ for every $K_{r+1}$-free $n$-vertex graph $G$ and if $\e(G)=\Ed$ then $G$ is isomorphic to $T_r(n)$, implying that the lemma holds when $\Ed - \e(G) \leq 0$.

For the induction step we assume $\Ed - \e(G) \geq 1$. Suppose first that $G$ is not a complete $r$-partite graph, i.e., there exists an $r$-partite graph $G'$ such that $G$ is obtained from $G'$ by deleting an edge between some pair of vertices $v_1,v_2 \in V(G')$. By Lemma~\ref{l:lowerbound} each map $\varphi'$ that  maps the ends of some edge of $H$ to $v_1$ and $v_2$ can be extended to at least $(1-\delta\v(H)^2)\v(G)^{\v(H)-2}$ homomorphisms in $\II(H,G')$, yielding a total of  $2\e(H)(1-\delta\v(H)^2)\v(G)^{\v(H)-2}$ homomorphisms  in $\II(H,G')-\II(H,G)$. Thus $$ \ii(H,G') -\ii(H,G) \geq 2\e(H)(1-\delta\v(H)^2)\v(G)^{\v(H)-2},$$
and   \begin{equation}\label{e:rpartite0} \ii(H,T_r(n)) - \ii(H,G') \geq 2\e(H)\left(\Ed - \e(G')\right)(1-3\delta\v(H)^3)\v(G)^{\v(H)-2},\end{equation}
by induction hypothesis, implying \eqref{e:rpartite}.

It remains to consider the case that $G$ is  complete $r$-partite. As $G$ is not a Tur\'an graph there exist parts $A,B'$ of the $r$-partition of $G$ such that $|A|\geq |B'|+2$. Let the graph $G'$ be obtained from $G$ by deleting some vertex of $a \in A$ and adding a new vertex $b$ such that $b$ is adjacent to every vertex of $G'$ except for vertices in $B'$, i.e., $G'$ has an $r$-partition with the parts $A$ and $B'$ replaced by parts $A -\{a\}$ and $B=B' \cup \{b\}$. Note that $\e(G')=\e(G)+|A|-|B|$.
By Lemma~\ref{l:rebalance} we have $$\ii(H,G') \geq \ii(H,G) + 2\e(H)(|A|-|B|)(1 - 3\delta\v(H)^3) \v(G)^{\v(H)-2},$$
Combining this inequality with \eqref{e:rpartite0}, which once again holds by the induction hypothesis, we obtain \eqref{e:rpartite}, as desired.
\end{proof}

\section{Proof of Theorem~\ref{thm:main2}}\label{sec:proof}

In this section we complete the proof of Theorem~\ref{thm:main2}. This result follows from

Theorem~\ref{thm:Furedi} and Lemma~\ref{l:rpartite}. The remaining technical difficulty is to show that the extremal graph $G$ must be $\delta$-dense for sufficiently small $\delta$.

\begin{proof}[Proof of Theorem~\ref{thm:main2}]
Theorem~\ref{thm:main2} trivially holds if $\e(H)=0$, so we assume $\e(H)\geq1$. Similarly, the theorem holds  if $n \leq r$ as  $T_r(n)=K_n$ for such $n$ and so we assume $n \geq r+1$. 

Suppose that $G$ is a $K_{r+1}$-free $n$-vertex graph that maximizes $\ii(H,G)$ among all such graphs. 
As every vertex of $T_r(n)$ has degree at least $n - n/r - 1 \geq (1-2/r)n$, the graph $T_r(n)$ is $(2/r)$-dense. Thus by Corollary \ref{l:iJG} we have 
\begin{align*}
    \ii(H,G)\geq \ii(H,T_r(n)) \geq  \left(1-\frac{2}{r}\v(H)^2\right)n^{\v(H)}.
\end{align*}
For  $v \in V(G)$ let $\ii(v)$ denote, for brevity, the number of homomorphisms in $\II(H,G)$ that contain $v$ in its image.
By averaging there exists $v_0 \in V(G)$ such that \begin{equation}\label{e:v_0} \ii(v_0) \geq \frac{\v(H)}{n}\ii(H,G) \geq \v(H)\left(1-\frac{2}{r}\v(H)^2\right)n^{\v(H)-1}.\end{equation} 
We now loosely upper bound $\ii(v)$  for arbitrary $v \in V(G)$ in terms of $\deg(v)$, as follows. There are  $\v(H)n^{\v(H)-1}$ (not necessarily injective) maps $\varphi: V(H) \to V(G)$ that contain $v$ in its image. On the other hand, if the map $\varphi: V(H) \to V(G)$ is such that  for some fixed edge $uu' \in E(H)$ we have $\varphi(u)=v$ and $\varphi(u')$ is not a neighbor of $v$, then $\varphi$ is not a homomorphism from $H$ to $G$. There are $(n -\deg(v))n^{\v(H)-2}$ such (not necessarily injective) maps, implying

\begin{equation}\label{e:u}\ii(v) \leq \v(H)n^{\v(H)-1} - (n-\deg(v))n^{\v(H)-2}.\end{equation} 
Construct a graph $G'$ from $G \setminus v$ by adding a copy of $v_0$, i.e.,~adding a vertex $v_1$ such that $v_1w \in E(G')$ if and only if $v_0w \in E(G)$ for $w \in V(G) -\{v\}$. Note that $G'$ is $K_{r+1}$-free as no clique contains $v_0$ and $v_1$ and replacing $v_1$ with $v_0$ in any clique of $G'$ gives a clique in $G$ of the same size. As at most $\v(H)^2n^{\v(H)-2}$ homomorphisms in $\II(H,G)$  contain both $v_0$ and $v$ in their image, we have
$$\ii(H,G) \geq \ii(H,G')\geq \ii(H,G) - \ii(v)+(\ii(v_0)-\v(H)^2n^{\v(H)-2}),$$
and so $\ii(v) \geq \ii(v_0) -\v(H)^2n^{\v(H)-2}$.
Thus by \eqref{e:v_0} and \eqref{e:u} we have
$$\v(H)n^{\v(H)-1} - (n-\deg(v))n^{\v(H)-2} \geq \v(H)\left(1-\frac{2}{r}\v(H)^2\right)n^{\v(H)-1}-\v(H)^2n^{\v(H)-2},$$
and so $$ \deg(v) \geq n -\frac{2}{r}\v(H)^3 n-\v(H)^2 \geq \s{1-\frac{3}{r}\v(H)^3}n \geq \s{1-\frac{1}{100\v(H)^6}}n$$
for every $v \in V(G)$, where we use $r \geq 300 \v(H)^9$. It follows that $G$ is $\frac{1}{100\v(H)^6}$-dense and $\e(G) \geq (1 - \frac{1}{100\v(H)^6})n^2/2$.

Theorem~\ref{thm:Furedi} gives an $r$-partite subgraph $G_0$ of $G$ such that $\e(G_0) \geq 2\e(G)-\e(T_r(n)).$ Note that every injective homomorphism $\varphi \in  \II(H,G) -\II(H,G_0)$ must map the ends of some edge $e \in E(H)$ to the ends of some edge $f \in G\setminus G_0$. There are $\e(H)$ choices of $e$, $(\e(G)-\e(G_0))$ choices of $f$, two choices of bijection between their ends, and at most $n^{\v(H)-2}$ choices of values of $\varphi$ on the remaining vertices of $H$. Putting this together yields 

\begin{equation}\label{e:upper}\ii(H,G_0) \geq  \ii(H,G) - 2\e(H)(\e(G)-\e(G_0))n^{\v(H)-2}.\end{equation}

We assume without loss of generality that $G_0$ is a maximal $r$-partite subgraph of $G$.
We have
\[ \e(G_0) \geq \e(G) - (\Ed-e(G)) \geq \e(G) -(n^2/2-\e(G)) \geq  \s{1 -  \frac{1}{50\v(H)^6}}\frac{n^2}{2}.\]
Let $A_1,A_2,\ldots,A_r$ be some partition of $V(G_0)$ into $r$ independent sets. As $2\e(G_0) \leq n^2 - \sum_{i=1}^r|A_i|^2$ we have $|A_i|^2  \leq \frac{1}{50\v(H)^6} \cdot n^2$ for every $i$, and so  $\max_{1 \leq i \leq r}|A_i| \leq \frac{n}{7\v(H)^3}$. By maximality of $G_0$, every edge in $E(G)-E(G_0)$ has both ends in some $A_i$, implying $$\deg_{G_0}(v) \geq \deg_G(v)-\max_{1 \leq i \leq r}|A_i| \geq \s{1- \frac{1}{100\v(H)^6} - \frac{1}{7\v(H)^3}}n \geq \s{1- \frac{1}{6\v(H)^3}}n$$
for every $v \in V(G)$. Thus $G_0$ is $\frac{1}{6\v(H)^3}$-dense and $\frac{1}{6\v(H)^3} < 1/4$. By 	Lemma~\ref{l:rpartite} we have
\begin{align*}
\ii(H,T_r(n)) &\geq  \ii(H,G_0) + 2\e(H)\s{1-3\s{\frac{1}{6\v(H)^3}} \v(H)^3} \left(\Ed - \e(G_0)\right)n^{\v(H)-2} \\
& \stackrel{\eqref{e:upper}}{\geq} \ii(H,G) - 2\e(H)(\e(G)-\e(G_0))n^{\v(H)-2}+\e(H)\left(\Ed - \e(G_0)\right)n^{\v(H)-2}\\
&= \ii(H,G) + \e(H)(\e(G_0)+\Ed-2\e(G))n^{\v(H)-2} \\ &\geq \ii(H,G),
\end{align*}
as desired.
\end{proof}

\section{Further Questions}

We conclude with a few questions that our current approach does not resolve. A classic strengthening of Tur\'an's Theorem due to Er\H{o}s and Simonovits~\cite{E67, S66} is the Stability Theorem: if $G$ is a $K_{r+1}$-free graph on $\ex(n, K_2, K_{r+1})-o(n^2)$ edges, then $G$ has edit distance $o(n^2)$ from the Tur\'an graph, which is to say $G$ can be transformed into the Tur\'an graph by adding and subtracting $o(n^2)$ edges. It is natural to ask if the generalized Tur\'an problem also exhibits stability.

We say that a graph $H$ is $F$-Tur\'an-stable if any $K_{r+1}$-free graph $G$ on $n$ vertices with $\mathcal{N}(H,T_r(n)) - o(n^{\v(H)})$ copies of $H$ has edit distance $o(n^2)$ from the Tur\'an graph.

\begin{question}\label{q:stability}
Fix any graph $H$ and let $r$ be large enough that $H$ is $K_{r+1}$-Tur\'an-good. Does it follow that $H$ is $K_{r+1}$-Tur\'an-stable?
\end{question}

Stability for generalized Tur\'an problems has been considered. Ma and Qiu~\cite{MQ20} proved that $K_r$ is $F$-Tur\'an-stable for $\chi(F) > r$. Gerbner~\cite{G22c} showed that if $\chi(F) = \chi(H)+1 = r+1$ and $H$ is both $K_{r+1}$-Tur\'an-good and $K_{r+1}$-Tur\'an-stable, then $H$ is $F$-Tur\'an-good. Our method shows such a graph $G$ has edit distance at most $O(n^2/r)$ from $T_r(n)$ but our techniques are not sufficient to show a sub-quadratic bound.

Another natural question is whether the generalized Tur\'an problem is monotonic in the following sense:

\begin{question}
Fix a graph $H$ and suppose $H$ is $K_r$-Tur\'an-good. Does it follow that $H$ is also $K_{r+1}$-Tur\'an-good? In other words, if $T_{r-1}(n)$ maximizes $\mathcal{N}(H,G)$ among $K_r$-free graphs, does it follow that $T_r(n)$ maximizes $\mathcal{N}(H,G)$ among $K_{r+1}$-free graphs?
\end{question}

One approach to proving monotonicity might be to start with a graph $G$ that is $K_{r+1}$-free and then remove edges to get $G'$ that is $K_r$-free. If this can be done in such a way that $\e(G)-\e(G') < \e(T_r(n)) - \e(T_{r-1}(n))$, then the methods described in this paper suggest
\[ \mathcal{N}(H,G) - \mathcal{N}(H,G') \le \mathcal{N}(H,T_r(n)) - \mathcal{N}(H,T_{r-1}(n)) \]
for sufficiently large $r$, and thus $\mathcal{N}(H,G') \le \mathcal{N}(H,T_{r-1}(n))$ is sufficient to demonstrate monotonicity.

Any $K_{r+1}$-free graph can be made $K_r$-free by removing at most $\frac{n^2}{r^2}$ edges as demonstrated by the following argument: If $S \subseteq E(G)$ is a minimal set of edges such that $G-S$ is $K_r$-free, then for each $e \in S$ there is an $r$-clique intersecting $S$ only at $e$; otherwise, $S-e$ is a smaller set intersecting every $K_r$. Thus, using Tur\'an's Theorem,
\[ \binom{r}{2}|S| \le \e(G) \le \binom{r}{2} \left(\frac{n}{r}\right)^2 \implies |S| \le \frac{n^2}{r^2}. \]
Unfortunately,
\[ \e(T_r(n)) - \e(T_{r-1}(n)) \approx \frac{n^2}{2r(r-1)} \]
is slightly smaller. This suggests a more delicate approach may be necessary where edges are added to $G$ in addition to being removed. Tracking the change in $\mathcal{N}(H,G)$ in such a process is beyond the scope of this paper.

\subsubsection*{Acknowledgements.} This research was partially completed at the \emph{Cross-community collaborations in combinatorics}  workshop (22w5107)  at the Banff International Research Station, 29 May-3 June 2022. We thank the organizers and other participants. The authors also thank D\'aniel Gerbner for his comments on an earlier draft.

\bibliographystyle{abbrv}
\bibliography{biblio}
\end{document}